\documentclass{ert-l}
\usepackage[all]{xy} 

\newtheorem{theorem}{Theorem}[section]

\newtheorem{proposition}[theorem]{Proposition}

\theoremstyle{definition}

\newtheorem{con}[theorem]{Conjecture}
\theoremstyle{remark}
\newtheorem{remark}[theorem]{Remark}

\numberwithin{equation}{section}
\newcommand{\Br}{\mathrm{Br}}

\newcommand{\F}{\mathcal{F}}

\newcommand{\Oc}{\mathcal{O}}

\newcommand{\Hom}{\mathrm{Hom}}
\newcommand{\Aut}{\mathrm{Aut}}
\begin{document}

\title[Reduction theorems for a conjecture]{Reduction theorems for a conjecture on bases in source algebras of blocks of finite groups}


\author{Tiberiu Cocone\c t}
\address{Babes-Bolyai University, Faculty of Economics and Business Administration, Str. Teodor Mihali, nr. 58-60, Cluj-Napoca, RO-400591, Romania}
\email{tiberiu.coconet@math.ubbcluj.ro}

\author{Constantin-Cosmin Todea}
\address{Technical University of Cluj-Napoca, Department of Mathematics, Str. G. Baritiu 25, Cluj-Napoca, RO-400027, Romania}
\address{Babes-Bolyai University, Department of Mathematics, Faculty of Mathematics and Computer
Science, Str. Mihail Kogalniceanu 1, 400084, Cluj-Napoca, Romania}
\email{constantin.todea@math.utcluj.ro}




\begin{abstract}
The aim of this short research note is to present  some results about a conjecture of Barker and Gelvin \cite[Conjecture 1.5 ]{BaGe} claiming that any source algebra of a $p$-block ($p$ is a prime) of a finite group has the unit group containing a basis stabilized by the left and right actions of the defect group. We obtain some reduction theorems for the existence of stable unital basis in source algebras of $p$-block algebras. Along the way we investigate this problem for the $p$-blocks of some finite simple groups.
\end{abstract}

\maketitle
\section{Introduction} \label{sec1}
In the context of $p$-modular representation theory ($p$ is a prime), source algebras of $p$-blocks, introduced by Puig in 1981, are smaller than the $p$-block algebras and are Morita equivalent to the corresponding $p$-block algebras. From now we shall call a $p$-block shortly a block, unless we need to emphasize the prime $p.$ The source algebra of a block preserves not only the equivalence class  of the module category but also invariants such as the defect groups, the category of subpairs, the  vertices and the sources of modules.The source algebra may be seen as the group representation version of the basic algebra associated with a finite dimensional algebra.

Let $\mathcal{\Oc}$ be a complete a complete principal ideal domain with identity element and with an algebraically closed residue field $k$ of prime characteristic $p$.  Any $\Oc$-algebra $\mathbf{A}$  considered in this paper has an  identity element $1_{\mathbf{A}}$ and we denote by $\mathbf{A}^{\times},$  the group consisting of invertible elements, that is the group of units. A basis for an algebra over $\Oc$ is an $\Oc$-basis as an $\Oc$-module. Such a basis is said to be \textit{unital} if every element is unital. The hypothesis on $\Oc$ implies that either $\Oc=k$ or $\Oc$ is a complete discrete valuation ring. We first present  a conjecture proposed by Barker and Gelvin for the case $\Oc=k.$ We assume familiarity with  the theory of group-acted algebras,  Brauer maps, fusions systems, blocks, source algebras, as in \cite{Libo1}, \cite{Libo2} and \cite{The}.
Let $G$ be a finite group such that $p$ divides the order of $G$  and  let $b$ be a block idempotent of $\Oc G$ with defect group $D$ (a $p$-subgroup of $G$) and $i\in (\Oc Gb)^D$ be a primitive idempotent such that $\Br_D^{\Oc G}(i)\neq 0,$ where $\Br_D^{\Oc G}$ is the Brauer map \cite[Definition 5.4.2]{Libo1} with respect to $D.$ The notation $A:=i\Oc Gi$ stands for the   interior $D$-algebra called the \textit{source algebra} of $b$ and $i$ is called a source idempotent of $b.$ The source algebra $A$ has a $D\times D$-stable  basis on which $D\times 1$ and $1\times D$ act freely. Following \cite[Section 6]{BaGe} this means that $A$ is a \textit{bifree bipermutation}  $D$-algebra. If  $A$  admits  a  $D\times D$-stable unital basis, we say that $A$ is a \textit{uniform source} $D$\textit{-algebra} of $\Oc Gb,$ see \cite[Theorem 1.2]{BaGe}.
 Let $\bar{b}\in Z(kG)$ denote the corresponding block of $b\in Z(\Oc G)$ via the canonical $\Oc$-algebra surjective homomorphism $\Oc G\rightarrow kG,$ see \cite[Proposition 37.4]{The}.
  It is well-known that $b$ and $\bar{b}$ share the same defect group $D$ and the same fusion system. Since  $A=i\Oc Gi$ is a source algebra of $b$ then \cite[Lemma 38.1]{The} yields a $D$-interior algebra surjective homomorphism $A\rightarrow \bar{A},$ where $\bar{A}:=i\Oc G i/J(\Oc) i\Oc G i$ is a source algebra of $\bar{b}.$ In \cite{BaGe} the authors propose the following conjecture for blocks over $\Oc,$ we  only state it here for  the case $\Oc=k$ and we exclude the blocks with trivial defect groups, since that situation is coverd by \cite[Proposition 6.6.3]{Libo2}.
\begin{con}\label{conj-BaGe}(\cite[Conjecture 1.5]{BaGe}, case $\Oc=k$) For any block idempotent $\bar{b}$ of $kG$ with non-trivial defect group, the source algebras  of $kG \bar{b}$ are uniform.
\end{con}
We fix a maximal $b$-Brauer  pair $(D,e)$ and denote the fusion system $\F_{(D,e)}(G,b),$ attached to the block $b,$ by $\F.$  This means that $b$ is an $\F$-block and we say, for shortness, that $b$ is an $\F$-\textit{reduction simple} block, if $D$ has no
nontrivial proper strongly $\F$-closed subgroups. The fusion system $\F$ is determined by the $\Oc D-\Oc D$-bimodule structure of $A$ and this fact can be reformulated in terms of permutation bases. By the work of Broto, Levi and Oliver it is known that any fusion system on $D$ has a so-called \textit{characteristic} $D-D$-biset, see \cite[Definition 8.7.9]{Libo2}; these bisets are important for the study of fusion systems and their cohomology. The source algebra $A$ always has a $D-D$-stable basis that is determined, up to an isomorphism of $D-D$-bisets, by the $\Oc D-\Oc D$-bimodule structure of $A.$ This biset is close to being a characteristic biset for $\F,$ but this is not known to hold in general. However, if this biset is included in $A^{\times}$ (hence \cite[Conjecture 1.3]{BaGe} is true) then this biset is also characteristic, see \cite[Proposition 8.7.11]{Libo2}.

In the first main result of this paper, proved  in  Section \ref{sec3}, we show that it is enough to solve Conjecture \ref{conj-BaGe} in order for  \cite[Conjecture 1.5]{BaGe} to hold.
\begin{theorem}\label{thmreduck} With the above notations if $\bar{A}$ is a uniform source $D$-algebra of $\bar{b} $ then $A$ is a  uniform source $D$-algebra of $b.$

\end{theorem}
If we do not impose on the idempotent $i\in (\Oc Gb)^D$ to be primitive but we require that for any subgroup $Q$ of $D$ there is a unique block
$e_Q$ of $\Oc C_G(Q)$ such that $\Br_Q^{\Oc G}(i)e_Q = \Br_Q^{\Oc G}(i),$ then we say that $i$ is an\textit{ almost source idempotent} of $b$ and $i\Oc Gi$ is an \textit{almost source algebra} of $b.$ As a matter of fact  \cite[Theorem 1.2]{BaGe} is given for this more general case. There is also  a similar conjecture for almost source algebras, \cite[Conjecture 1.3]{BaGe}, but we are not interested for the moment about this conjecture. We denote by $\Aut_G(D)$ the group of all group automorphisms of $D,$ given by conjugation with elements of $G.$

In the second main theorem of this paper we obtain  a reduction result saying that if we solve Conjecture \ref{conj-BaGe} for all blocks, with non-trivial defect groups, of all quasi-simple finite groups then this conjecture is true for all $\F$-reduction simple blocks, of finite groups $G$ having defect groups $D$ such that $\Aut_G(D)$ is a $p$-group.

\begin{theorem}\label{thm:12}
If for any finite  quasi-simple group $L$, for any  block idempotent $\bar{d}$ of $kL$ with $D$ a non-trivial defect group, there is uniform  source $D$-algebra of $kL\bar{d}$, then for any finite group $G$ and for any $\F$-reduction simple block  $\bar{b}$ of $kG$ with non-trivial defect group $D$ (where $\F:=\F_{(D,e)}(G,b)$ and $(D,e)$ is a maximal $\bar{b}$-Brauer pair), such that $\Aut_G(D)$ is a $p$-group, there is a uniform  source $D$-algebra  of $kG\bar{b}.$ In particular, any source algebra will be uniform.
\end{theorem}

The proof of Theorem \ref{thm:12} is given in Section \ref{sec3}. More precisely, the method of proof for the above theorem relies on the following deduction. Assuming that  there is  a finite group $G$ and an $\F$-reduction simple block $\bar{b}$ of $kG$ 
with non-trivial defect group $D,$ such that $\Aut_G(D)$ is a $p$-group and no source algebra of $kG\bar{b}$ is uniform, then there is quasi-simple finite group $L$ and there is a block $\bar{d}$ of $kL$ with non-trivial defect group such that no source algebra of $kL\bar{d}$ is uniform.

In the last theorem we verify Conjecture \ref{conj-BaGe} for most blocks  of the first  sporadic groups. 
\begin{theorem}\label{thm14-sporadic} Let $G$ be a finite group such that $p$ divides the order of $G$ and let $\bar{b}$ be a block of $kG.$ Assume that one of the following statements is true.
\begin{itemize}
\item[(i)] $G$ is a sporadic Mathieu group but, when $p=3$ and $G\in \{M_{12}, M_{24}\}$ or when $p=2$ and $G\in \{M_{11},M_{12}, M_{22}, M_{23}, M_{24}\},$ $\bar{b}$ is not the principal block;
\item[(ii)] $G\in\{J_1,J_2,J_4\}$  but, when $p=2$ and $G\in\{J_2,J_4\}$ or when $p=3$ and $G=J_4,$ the block $\bar{b}$ is not principal block and it is not the unique $3$-block of maximal defect group $3_+^{1+2}.$
\end{itemize}
Then any source algebra of $kG\bar{b}$ is uniform.
\end{theorem}
The proof of Theorem \ref{thm14-sporadic} is given in Section \ref{sec4}. As a notational aspect, if $\mathbf{A}$ is any $G$-algebra and $H_{\alpha}$ is a pointed group, where $H$ is a subgroup of $G,$ the \textit{localization} of $\mathbf{A}$ with respect to $H_{\alpha}$ is the $\Oc$-algebra $\mathbf{A}_{\alpha},$ that is $\mathbf{A}_{\alpha}=i\mathbf{A}i,$ where $i\in\alpha;$ see \cite[Section 13]{The} for more details. 
\section{Uniform source algebras of blocks of normal subgroups}\label{sec2}
The purpose of this section is to prove a crucial argument needed in the proof of Theorem \ref{thm:12} (Claim 2). We allow blocks over $\Oc$ although we need the results for blocks over $k.$ In this section let $N$ be a normal subgroup of $G,$ let $c$ be a block idempotent of $\Oc N$ such that $c\in (\Oc N)^G$ and let $b$ be a block idempotent of $\Oc G.$ 

Let $D_{\delta}\leq N_{\{c\}}$ be a defect pointed group of $c$ in $N.$ We recall the notation of 
$[G/N]$ as a system of representatives of left cosets of $N$ in $G$ and introduce
 $$\mathbf{A}:=(\Oc G)_{\delta}=\oplus_{x\in [G/N]}\mathbf{A}_{\bar{x}}$$ as $G/N$-graded, $G$-interior algebra, where $\mathbf{A}_{\bar{x}}:=j(\Oc Nx)j,$ for some  $j\in\delta$ and $x\in [G/N].$ 
\begin{proposition} \label{prop 2.2}If $\mathbf{A}_1$ is a uniform source $D$-algebra of $c$, then for any $x\in[G/N],$ $\mathbf{A}_{\bar{x}}$  has a $D\times D$-stable unital basis. In particular $\mathbf{A}$ has a $D\times D$-stable unital basis.
\end{proposition} 
\textit{Proof}
By the Frattini argument $G=N_G(D_{\delta})N,$ hence any representative $x\in [G/N]$ may be chosen such that $x\in N_G(D_{\delta}).$ After this, by letting $j\in\delta,$ we have
${}^xj={}^{a_x}j$ for some $a_x\in ((\mathbf{A}_1)^D)^{\times},$ equivalently $a_x^{-1}xj=ja_x^{-1}x.$ 

The decomposition $\mathbf{A}=\bigoplus_{x\in [G/N]}\mathbf{A}_{\bar{x}}$ is well known and any $\mathbf{A}_{\bar{x}}$ verifies $$\mathbf{A}_{\bar{x}}=\mathbf{A}_1\cdot a_x^{-1}x,$$ for any $\bar{x}\in G/N.$ Let $x\in[G/N],$ with $x\in N_G(D_{\delta}).$ If $\Omega$ is a $D\times D$-stable unital basis of $\mathbf{A}_1$  then $$\Omega_x:=\{\omega a_x^{-1}x | \omega \in\Omega\}$$ is
 a $D\times D$-stable unital basis of $\mathbf{A}_{\bar{x}}.$ The elements of each $\Omega_x$ are clearly linearly independent. For any $\omega \in\Omega,u_1,u_2\in D$ we have
 $$u_1(\omega a_x^{-1}x) u_2=u_1\omega ( {}^xu_2)( a_x^{-1}\ x)\in\Omega_x.$$

\begin{proposition}\label{thm22} Let $N$ be a normal subgroup of $G$ such that $G/N$ is a $p'$-group. Let $b$ be a block of  $\Oc G$ such that $b\in Z(\Oc N)$ and remains a block of $\Oc N.$ Let $D_{\delta}$ be a defect pointed group of $N_{\{b\}}$ and we assume that $(\Oc N)_{\delta}$ is a uniform source $D$-algebra and that $\Aut_G(D)$ is a $p$-group.
Then $(\Oc G)_{\delta}$ is a uniform source $D$-algebra of $b$ (as a block of $\Oc G$). Moreover, any source $D$-algebra of the block $b$ of $\Oc G$ is uniform.
\end{proposition}

\begin{proof}
Let $(D,\hat{e})$ and $(D,e)$ denote maximal $b$-Brauer pairs, one for $b$ as a block of $\Oc N$ and one for $b$ as a block of $\Oc G,$  respectively. Choose $j\in \delta$ such that $(\Oc N)_{\delta}=j\Oc Nj$ is a source algebra of the block $b$ of $\Oc N.$ Recall that $C_G(D_{\delta})$ is the group of elements $g\in C_G(D)$ such that ${}^g\delta=\delta$ and similarly for $N_G(D_{\delta}).$ Let $x\in N_G(D_{\delta}).$  

We show next that $N_G(D_{\delta})=N_N(D_{\delta})C_G(D_{\delta}),$ hence  the conjugation by $x$ from $D$ to ${}^x D$ is an $\mathcal{F}_{(N,b)}(D,\hat{e})$-isomorphism. Let $\Psi:N_G(D_{\delta})\rightarrow \Aut_G(D)$ be the group homomorphism given by $\Psi(g)=c_g, g\in N_G(D_{\delta}),$ where $c_g$ is the conjugation by $g.$ Since $\mathrm{Ker}(\Psi)=C_G(D_{\delta})$ it follows that $N_G(D_{\delta})/C_G(D_{\delta})$ is a $p$-group, hence $N_G(D_{\delta})/N_N(D_{\delta})C_G(D_{\delta})$ is also a $p$-group. But,  by the Frattini argument $G=N_G(D_{\delta})N,$ it follows that $G/N\cong N_G(D_{\delta})/N_N(D_{\delta})$ is a $p'$-group, hence $N_G(D_{\delta})/N_N(D_{\delta})C_G(D_{\delta})$ is also a $p'$-group, thus proving the claim.

Since $(\Oc N)_{\delta}$ is uniform,   by \cite[Theorem 1.2]{BaGe} the $\mathcal{F}_{(N,b)}(D,\hat{e})$-isomorphism $D\to {}^{x}D$ determines a unit $r\in ((\Oc N)_{\delta})^{\times}$ such that 
$$({}^{x^{-1}}l) (r^{-1}) l^{-1}=r^{-1},$$ for any $l\in D.$  The last equality is equivalent to $$l(r^{-1})( {}^{x} l^{-1})=r^{-1},$$ for any $l\in D.$
Similarly  to the proof of Proposition \ref{prop 2.2} there is $a_x\in ((j\Oc N j)^D)^{\times}$ such that  $r^{-1}a_x^{-1}x$ is a homogeneous unit of $(j\Oc Gj)^D.$ It follows that the $G/N$-graded algebra $(j\Oc Gj)^D$ is actually a strongly $G/N$-graded algebra. Then \cite[Corollary 1.6.8. and Theorem 1.5.14.]{Mar} assure us that $(j\Oc Gj)^D$ is a local ring, since $G/N$ is of  $p'$-order; for a more recent exposition of this argument see \cite[Lemma 6.4]{HuLiZh}. Next, according to \cite[Theorem 4.4.4]{Libo1}, we obtain that $j\Oc Gj$ is a source algebra of $b$ as a block of $\Oc G.$  Finally, note that Proposition \ref{prop 2.2} implies that $j\Oc Gj$  is  uniform. 
\end{proof}
\section{Proofs of Theorem \ref{thmreduck} and \ref{thm:12}} \label{sec3}
\begin{proof} \textbf{(of Theorem \ref{thmreduck}).}
Let $\mathcal{F}$ be the fusion system of $b$ and $\bar{b}$ and we fix $\phi:Q\rightarrow P$ a group isomorphism such that $\phi\in \Hom_{\F}(Q,P),$ where $Q,P\leq D.$

We will show that there are surjective $\Oc$-module homomorphisms $A(\phi)\rightarrow \bar{A}(\phi)$ and $A(\phi^{-1})\rightarrow \bar{A}(\phi^{-1})$   and  that the Brauer quotients $A(P)$ and $\bar{A}(P)$ are isomorphic  $k$-algebras; see \cite[Section 4]{BaGe} for $A(\phi), A(\phi^{-1}).$ 

Indeed, the surjective $\Oc$-algebra homomorphism $(\Oc G)^P\rightarrow (kG)^P,$ admitting the kernel $J(\Oc)(\Oc G)^P,$ determines the $k$-algebra isomorphisms
between the Brauer quotients

$$ (\Oc G)^P/(J(\Oc)(\Oc G)^P+ \sum_{T<P}(\Oc G)_T^P)\cong (kG)^P/\sum_{T<P}(kG)_T^P\cong kC_G(P).
$$

So  the Brauer homomorphism  $(\Oc G)^P\rightarrow (\Oc G)(P)$  can be obtained as the following composition of surjective $\Oc $-algebra homomorphisms
$$(\Oc G)^P\to (kG)^P\to (k G)(P).$$
The restriction to $A$ of the above composition determines the surjective  $\Oc$-algebra homomorphism 
$A^P\rightarrow \bar{A}(P).$
The kernel of this map is $$i(J(\Oc)(\Oc G)^P+\sum_{T<P} (\Oc G)_T^P)i=J(\Oc)A^P+\sum_{T<P}A_T^P,$$ where $i$ is our fixed source idempotent determining $A.$ It follows that
$$A(P)=A^P/(J(\Oc)A^P+\sum_{T<P}A_T^P)\cong \bar{A}(P),$$ see also
\cite[Proposition 5.6]{PuigNil}. 
Similarly, the Brauer homomorphism (of $\Oc$-modules) with respect to $\Delta(\phi)$ may be viewed as the following composition of $\Oc$-module surjective homomorphisms
$$(\Oc G)^{\Delta(\phi)}\rightarrow (kG)^{\Delta(\phi)}
\rightarrow (kG)(\phi),$$
that restricts to the surjective homomorphism of $\Oc $-modules 
\[A^{\Delta(\phi)}\rightarrow \bar{A}(\phi).\]
Hence we have a surjective homomorphism $A(\phi)\rightarrow \bar{A}(\phi),$ of $k$-vector spaces. Similar arguments hold for $\phi^{-1}$ in place of $\phi $ and $Q$ in place of $P.$ Finally, since $\bar{A}$ is uniform and the vertical maps in the  next commutative diagram are surjective 

\begin{equation*}
\xymatrix{ \bar{A}(\phi)\times \bar{A}(\phi^{-1}) \ar@{->>}[r]&\bar{A}(P) \\
   A(\phi)\ar@{->>}@<-5ex>[u] \ar@{->>}@<5.5ex>[u] \times A(\phi^{-1})\ar[r]& A(P)\ar@{>->>}[u]}
\end{equation*}
by \cite[Theorems 5.1 and 1.2]{BaGe}, it follows  that $A$ is uniform.
\end{proof}
The following remark, that includes a much shorter proof of Theorem \ref{thmreduck}, was pointed to us by the referees.
\begin{remark}
By standard arguments, if $M$ is
permutation $\Oc P$-module for some finite $p$-group $P$, then any $P$-stable
basis $Y$ of the $kP$-module $M/J(\Oc)M$ lifts to some $P$-stable basis of $M$ itself.
(Take a permutation basis $X$ of $M$, its image $Y'$ in $M/J(\Oc)M$ is isomorphic
to $Y$ as a $P$-set, so we have a $kP$-automorphism of $M/J(\Oc)M$  sending $Y'$ to $Y,$
and by standard facts, this automorphism lifts to an $\Oc P$-automorphism of $M,$
and then the image of $X$ under this automorphism of $M$ lifts $Y$).
Applied to a source algebra and $D \times D$ instead of $P,$ using that any inverse 
image of an invertible element over $k$ is invertible over $\Oc$ by Nakayama's Lemma,
it follows that it is equivalent for a source algebra over $k$ or $\Oc$ to have a permutation
basis of invertible elements.
\end{remark}
\begin{proof} \textbf{(of Theorem \ref{thm:12}).} We assume that $G$ is a finite group of minimal order with respect to having an $\F$-reduction simple block $\bar{b}$ of non-trivial defect  group $D$ such that no source $D$-algebra of $kG\bar{b}$ is uniform. We divide the  proof in several claims.

\textit{Claim 1.} For any $H,$ a normal subgroup of $G,$ such that $p$ divides the order of $H$ and such that $\bar{c}$ is a block of $kH$ covered by $\bar{b},$ it follows that $\bar{c}$ is $G$-stable. \\
This is true by a well-known Clifford theoretic fact and by the minimality of $|G|,$ see \cite[Proposition 2.13]{art:Kess2006}.

\textit{Claim 2.} The group $G$ is generated by all conjugates of $D.$ \\ For showing this, let $H:=<{}^gD| g\in G>,$ which  is a normal subgroup of $G.$ Using Claim $1$, let $\bar{c}$ be the unique block of $kH$ covered by $\bar{b}.$ We know that $G$ acts by conjugation on $kH\bar{c}$ and hence we  have a group homomorphism from $G$ to the group of outer automorphisms (as $k$-algebras) of $kH\bar{c}.$  We denote by $K$ the normal subgroup of $G$ obtained as the kernel of this group homomorphism, see \cite[3.2]{art:Kess2006} or \cite[Section 5]{art:Kuels1995}. By Claim 1 we consider $\bar{c_1}$ the unique block of $kK$ covered by $\bar{b}.$ By \cite[Theorem, p.303 ]{art:Kuels1995} we know that $G/K$ is a $p'$-group, $\bar{c_1}$ is in fact $\bar{b}$ (having the same defect group $D$ in $K$) which remains a block of $kK$ and  with the block algebra $kG\bar{b}$ being a  crossed product  of $G/K$ with $kKb.$ Making suitable choices of source idempotents and applying Proposition \ref{thm22} we obtain that $kK\bar{b}$ has no uniform source algebra. By the  minimality assumption we obtain that $G=K.$ Since $G=K$ and $\bar{b},\bar{c}$  are blocks of $kG,kH,$ respectively, having a common defect group $D,$ by \cite[Theorem 7]{art:Kuels1990} we get that $kG\bar{b}$ and $kH\bar{c}$ have isomorphic source algebras. By contradiction, it follows that $kH\bar{c}$ has no uniform source $D$-algebra (see \cite[Theorem 1.7]{CoTo} for the more general case of basic Morita equivalences); hence again the minimality assumption forces the equality $G=H.$

\textit{Claim 3.} For any proper normal subgroup $N$ of $G,$ if $\bar{b'}$ is a block of $kN$ covered by $\bar{b},$ then $\bar{b'}$ is with defect zero.\\
By Claim 1 we assume that $\bar{b'}$ is $G$-stable and then by \cite[Theorem 6.8.9, (ii)]{Libo2} we have that $D\cap N$ is a defect group of $\bar{b'}.$ If $D\cap N\neq \{1\},$ since $D\cap N$ is always strongly $\F$-closed and $\bar{b}$ is $\F$-reduction simple, it follows that $D\cap N=D.$ That is $D\subseteq N,$ hence $D$ and all its conjugates are in $N.$ Using Claim 2 it follows that $N=G,$ a contradiction.

\textit{Claim 4.} Under  the assumptions of Claim 3, by a variation of second Fong's reduction (see \cite[Proof of Theorem 3.1]{art:Kess2006}, \cite[Proposition 4.7]{HiLa} or \cite[Proposition 2.1]{KeKoLi}), there is a central $p'$-extension
\[
\xymatrix{
1\ar[r] & Z\ar[r]&\tilde{G}\ar[r]& G/N\ar[r]&1,}
\]
and there is an $\F$-block $\bar{c}$ of $k\tilde{G}$ such that $kG\bar{b}$ is basically Morita equivalent to $k\tilde{G}\bar{c}.$  So, by \cite[Theorem 1.7]{CoTo} $k\tilde{G}\bar{c}$ has no uniform source algebra.

\textit{Claim 5.} We assume now that $N$ is maximal proper subgroup of $G.$ Applying Claim 4 in this case, we get a central $p'$-extension

\[
\xymatrix{
1\ar[r] & Z\ar[r]&\tilde{G}\ar[r]& \bar{G}\ar[r]&1,}
\]
where $\bar{G}:=G/N$ is a simple finite group. It is clear that $Z(\tilde G)/Z \unlhd \tilde G/Z \cong \overline G,$ hence $Z=Z(\tilde G)$ is a cyclic $p'$-subgroup. 
Again by Claim 4, there is an $\mathcal F$-block $\bar{c}$ of $k \tilde G$ with no uniform source algebra. Let $L=[\tilde G, \tilde G]$.\\ 
First note that if $\tilde G/Z \cong \overline G$ is a cyclic group, then $\tilde G$ must be abelian, see proof of \cite[Lemma 3.4]{serwene}, hence $L=1$ in that case. But if $ \tilde{G}$  is abelian, it follows that $Z=\tilde{G},$ hence $\bar{G}=1,$ a contradiction. \\
We have that $\tilde G=LZ$, $L$ is a quasi-simple group and $Z(L) \subseteq Z $ is a cyclic $p'$-group, see again proof of \cite[Lemma 3.4]{serwene}. There is a short exact sequence:

\[
\xymatrix{
1\ar[r] & K\ar[r]& L \times Z \ar[r]^{\pi} &\tilde{G} \ar[r]&1,}
\]
where $\pi (e,z)=ez$, for any $e\in L,z\in Z$ and $$K:=\ker \pi=\{ (e,e^{-1}) \mid e \in L \cap Z \}.$$
By \cite[Lemma 2.10(i)]{fk}, for a block $\bar{c}$ of $k(L \times Z)/K$, there is a unique block $\bar{c'}$ of $k(L \times Z)$ such that $\bar{c}$ is dominated by $\bar{c'}$, that is $\pi(\bar{c'}) \neq 0;$ here $\pi(\bar{c'})$ is an idempotent of $Z(k \tilde G)$ and $\pi$ is the $k$-algebra surjection $\pi: k(L \times Z) \rightarrow k \tilde G$ (by abuse of notation) induced by the canonical surjection $$\pi:L \times Z \rightarrow \tilde G,\quad \tilde G \cong (L \times Z)/K.$$ Since $K$ is a $p'$-subgroup and $\pi(\bar{c'}) \neq 0,$ by \cite[Lemma 2.10(ii)]{fk} it follows that $\pi(\bar{c'})=\bar{c}$ and $k \tilde G \bar{c} \cong k(L \times Z)\bar{c'}$ as $k$-algebras. Thus, we obtained a block $\bar{c'}$ of $k(L \times Z)$ with no uniform source algebra. 

Finally,  utilizing the well-known fact that if $\bar{c'}$ is a block of $k(L \times Z),$ then there is a block of $kL$ and a block of $kZ$ (hence with defect zero) such that $k(L \times Z)\bar{c'}$ decomposes into the tensor product of these two block algebras,
it follows that canonical $k$-algebra surjection $\mu:k(L\times Z)\rightarrow kL$ has the property $\mu(\bar{c'})\neq 0.$ Again, using \cite[Lemma 2.10(i)]{fk}, there is a block $\bar{d}$ of $kL$ such that $\mu(\bar{c'})=\bar{d}$ and $k(L \times Z)\bar{c'}\cong kL\bar{d}$ as $k$-algebras; hence, we can identify a source algebra of $\bar{c'}$ with a source algebra of $\bar{d}.$ In conclusion we obtained a quasi-simple finite group $L$ and a block $\bar{d}$ of $kL,$ that has no uniform source algebra.
\end{proof}

\section{Uniform source algebras  of blocks of sporadic groups}\label{sec4}

\begin{proof}\textbf{(of Theorem \ref{thm14-sporadic}).} For each block that  has a trivial defect group we apply \cite[Proposition 6.6.3]{Libo2}.  

(i). For sporadic Mathieu groups the $p$-block structure of $kG$ for the prime di-\\visors $p$ of $|G|$, in particular the number of blocks and their defect, is recently collected in \cite[Table 5]{Mur}. As the author claims, these data can be verified using GAP. Note that if $p\in\{5,7,11,23\}$ all blocks with non-trivial defect groups of all sporadic Mathieu groups have cyclic defect group, hence our conclusion is true by \cite[Proposition 1.6]{BaGe}. When $p\in \{2,3\}$ again all non-principal blocks (with non-trivial defect group) of the Mathieu groups are of cyclic defect group, excepting the case $p=2$ and $G=M_{12},$ which has a non-principal $2$-block of Klein four defect group. We apply \cite[Proposition 1.6]{BaGe} to obtain the conclusion for these cases. For $p=3$ by \cite[Tables 5,7,8]{Mur} we note that the principal $3$-blocks of $M_{11}, M_{22}$ and $ M_{23}$ have defect group $C_3\times C_3$, hence again \cite[Proposition 1.6]{BaGe} assure us that these blocks have uniform source algebras. We are left with the principal blocks of all Mathieu groups for $p=2$ and with the principal blocks of $M_{12}$ and  $M_{24}$ if $p=3.$\\

(ii). Let $G=J_1$. Recall that  $$|J_1|=175560=2^3 \cdot 3 \cdot 5 \cdot 7 \cdot 11 \cdot 19.$$
If $p$ is odd, then all blocks that have a non-trivial defect group must have a cyclic defect group. If $p=2$ using the character table from \cite{Janko} and the details given in the Introduction of \cite{LandMich}, it is easy to see that all $2$-blocks of $J_1$ have defect zero or one, except for the principal $2$-block whose defect group is $C_2 \times C_2 \times C_2$. Thus, all blocks with non-trivial defect groups have abelian defect groups and we apply  \cite[Proposition 1.6]{BaGe}.

Let $G=J_2.$  We know that  $$|J_2|=604800=2^7 \cdot 3^3 \cdot 5^2 \cdot 7.$$ In this case we use the description of the modular representation theory given in \cite{sin}.  If $p\in\{5,7\}$ then all blocks with non-trivial defect group have abelian defect group. The same is true if $p\in\{2,3\}$ for all non-principal blocks with non-trivial defect groups. \\
If $p=3$ the principal block $b_0$  of $J_2$ has defect group the Sylow $3$-group $3_{+}^{1+2},$ (the extraspecial group of order $27$ and exponent $3$) and  the fusion system $\F_0=\F_{3_{+}^{1+2}}(J_2).$ By  \cite[Remark 1.4, Lemma 3.2]{RuVi} it follows that the only $\F_0$-centric, $\F_0$-radical subgroup is $3_{+}^{1+2}.$ We apply now \cite[Proposition 4.5, Part I]{AKO} to obtain that $\F_0=N_{\F_0}(3_{+}^{1+2}),$ hence again using \cite[Proposition 1.6]{BaGe} we are done with this case.

For $G=J_4$ we have $$|J_4|=2^{21}\cdot 3^3\cdot 5\cdot 7\cdot 11^3\cdot 23\cdot 29\cdot 31\cdot 37\cdot 43.$$ The descriptions of the blocks  of this group with their defect groups appears  in \cite[Lemma 5.1]{An} and \cite[Lemma 3.2]{Kos}. If $p\notin\{2,3,11\}$ then any block of $kG$  has cyclic or trivial defect group and we are done. For $p=2$ the principal block $kG$ is the unique block of positive defect (see \cite[Lemma 5.1 (b)]{An}).

For $p=3$ \cite[Lemma 5.1 (a)]{An} assure us that there are $7$ blocks of positive defect, giving the decomposition
$$kG=kG\bar{b}_0 \times kG\bar{b}_1 \times kG\bar{b}_2 \times kG\bar{b}_3\times kG\bar{b}_4\times kG\bar{b}_5\times kG\bar{b}_6,$$
The block $\bar{b}_2$ is the unique block  with defect group an elementary abelian $3$-group of order $9.$ All the other blocks $\bar{b}_i,i\in\{3,4,5,6\}$ have cyclic defect group $C_3.$

For $p=11$ we know by \cite[Lemma 3.2]{Kos} that the principal block is the unique block of positive defect, see also \cite[Table 6, p.464]{Bla}. The defect group of the principal block is the Sylow $11$-subgroup $11_+^{1+2}.$ Since the fusion system of this principal block is  $\F'_0=\F_{11_{+}^{1+2}}(J_4),$ by \cite[Remark 1.4, Lemma 3.2]{RuVi} it follows that the only $\F'_0$-centric, $\F'_0$-radical subgroup is $11_{+}^{1+2}.$ Now, \cite[Proposition 4.5, Part I]{AKO} allow us to obtain that $\F'_0=N_{\F'_0}(11_{+}^{1+2}),$ hence again \cite[Proposition 1.6]{BaGe} concludes our statement.
\end{proof}

\textbf{Acknowledgments.} "This work was supported by a grant of the Ministry of Research, Innovation and Digitization, CNCS -UEFISCDI, project number PN-IV-P1-PCE-2023-0060, within PNCDI IV."

\bibliographystyle{amsplain}


\end{document}